\documentclass[a4paper,12pt]{article}
\usepackage{amsfonts}
\usepackage{amsmath}
\usepackage[mathscr]{eucal}
\usepackage{amssymb}
\usepackage{latexsym}
\usepackage{amsthm}
\usepackage{amscd}
\usepackage{epsf}
\usepackage{graphicx}
\usepackage{color}
\usepackage{enumerate}

\newtheorem{theorem}{Theorem}[section]

\newtheorem{corollary}[theorem]{Corollary}
\newtheorem{lemma}[theorem]{Lemma}

\newtheorem{definition}[theorem]{Definition}
\newtheorem{algorithm}[theorem]{Algorithm}

\newtheorem{remark}[theorem]{Remark}

\newcommand{\rd}{\,\mathrm{d}}

\newcommand{\bsf}{\boldsymbol{f}}
\newcommand{\bsg}{\boldsymbol{g}}
\newcommand{\bsh}{\boldsymbol{h}}

\newcommand{\bsx}{\boldsymbol{x}}
\newcommand{\bsy}{\boldsymbol{y}}
\newcommand{\bsz}{\boldsymbol{z}}

\newcommand{\bszeta}{\boldsymbol{\zeta}}
\newcommand{\bszero}{\boldsymbol{0}}

\newcommand{\Pcal}{\mathcal{P}}

\newcommand{\icomp}{\texttt{i}}

\def\bN{\mathbb N}
\def\bZ{{\mathbb Z}}

\def\bR{{\mathbb R}}
\def\bC{{\mathbb C}}
\def\bF{{\mathbb F}}

\def\F2{{\mathbb F}_2}

\def\bsx{{\boldsymbol{x}}}
\def\bsy{{\boldsymbol{y}}}

\newcommand{\bbX}{\mathbb{X}}
\newcommand{\Fb}{\mathbb{F}_b}

\newcommand{\Lfield}{\Fb((x^{-1}))}
\newcommand{\Pring}{\Fb[x]}
\newcommand{\Lring}{\Fb[[x^{-1}]]}
\newcommand{\Ptset}[1][\bsf]{\mathcal{P}_{{#1},d}}
\newcommand{\Ucube}[1][d]{\mathbb{U}^{#1}}
\newcommand{\TrBox}[2][d]{\Ucube[#1]_{#2}}

\newcommand{\Ppart}[1]{\left(#1\right)_{\mathrm{poly}}}

\newcommand{\Res}{\mathrm{res}}

\newcommand{\Dnet}[2][{}]{\mathcal{D}_{#1}(#2)}
\newcommand{\Dlat}[1]{{#1}^*}
\newcommand{\DinL}[1]{{#1}^\perp}
\newcommand{\Dgp}[1]{{#1}^{\wedge}}

\newcommand{\minNRT}[2][{}]{\mathrm{minNRT}_{#1}(#2)}
\newcommand{\omegab}{\omega_b}
\newcommand{\PtoPmap}[1]{\theta_{#1}}
\newcommand{\bNm}{\mathbb{Z}_{\leq 0}}
\newcommand{\trunc}{\mathrm{tr}}
\newcommand{\adnum}{M}
\newcommand{\Diagm}[1]{D_{#1}}

\newenvironment{enuroman}{\begin{enumerate}[\normalfont (i)]}{\end{enumerate}}

\begin{document}

\title{Digital net properties of a polynomial analogue of Frolov's construction}
\author{Josef Dick\thanks{The research of J. Dick was supported under the Australian Research Councils Discovery Projects funding scheme (project number DP150101770).}, Friedrich Pillichshammer\thanks{The research of F. Pillichshammer was supported by the Austrian Science Fund (FWF): Project F5509-N26 which is a part of the Special Research Program ``Quasi-Monte Carlo Methods: Theory and Applications''.}, Kosuke Suzuki
\thanks{The research of K. Suzuki was supported by Australian Research Council’s Discovery Projects funding scheme (project number DP150101770), JST CREST and JSPS Grant-in-Aid for JSPS Fellows (No. 17J00466).}, \\
Mario Ullrich, Takehito Yoshiki
\thanks{The research of T. Yoshiki was supported by Australian Research Council’s Discovery Projects funding scheme (project number DP150101770), JST CREST and JSPS Grant-in-Aid for JSPS Fellows (No. 17J02651).}}
\date{\today}

\maketitle

\begin{abstract}
Frolov's cubature formula on the unit hypercube has been considered important
since it attains an optimal rate of convergence for various function spaces.
Its integration nodes are given by shrinking a suitable full rank $\bZ$-lattice 
in $\bR^d$ and taking all points inside the unit cube.
The main drawback of these nodes is that they are hard to find computationally, 
especially in high dimensions.In such situations, quasi-Monte Carlo (QMC) rules based on digital nets have proven to be successful. However, there is still no construction known that leads to QMC rules which are optimal in the same generality as Frolov's.

In this paper we investigate a polynomial analog of Frolov's cubature formula, 
which we expect to be important in this respect.
This analog is defined in a field of Laurent series with 
coefficients in a finite field. A similar approach was previously studied in [M.~B.~Levin.
Adelic constructions of low discrepancy sequences. Online Journal of Analytic Combinatorics. Issue 5, 2010.].

We show that our construction is a $(t,m,d)$-net, which also implies 
bounds on its star-discrepancy and the error of the corresponding 
cubature rule.
Moreover, we show that our cubature rule is a QMC rule, 
whereas Frolov's is not, and provide an algorithm to determine its 
integration nodes explicitly.

To this end we need to extend the notion of $(t,m,d)$-nets 
to fit the situation 
that the points can have infinite digit expansion 
and develop a duality theory. 
Additionally, we adapt the notion of admissible lattices to our setting 
and prove its significance.
\end{abstract}



\section{Introduction}\label{sec:Introduction}
In this paper we consider numerical integration on the $d$-dimensional unit cube
\[
\int_{[0,1]^d} \psi(\bsx) \rd\bsx
\]
approximated by an algorithm using $n$ function evaluations as
\[
\sum_{i=1}^n w_i\psi(\bsx_i) \quad \text{for $w_i \in \bR$, $\bsx_i \in [0,1]^d$}.
\]
If the weights satisfy $w_i = n^{-1}$ for all $i$, the algorithm is called a quasi-Monte Carlo (QMC) rule.
For QMC rules, lattice rules (i.e., QMC rules using integration lattices)
and digital net rules (i.e., those using digital nets) have been mainly considered,
see the books \cite{Dick2010dna,Niederreiter1992rng,Sloan1994lmm} and the references therein.
One intensively studied class of digital net rules is polynomial lattice rules
first proposed in \cite{Niederreiter1992ldp}.
Polynomial lattice rules are a polynomial analog of lattice rules,
where polynomial analog means that $\bR$ and $\bZ$ lattice rules are replaced
by a field of Laurent series $\Fb((x^{-1}))$ and a ring of polynomials $\Fb[x]$.

On the other hand, one important non-QMC rule is Frolov's cubature formula, see Section~\ref{sec:Frolov}.
The fascinating property of this cubature rule is that, although the construction is fixed, it attains an optimal rate of convergence for various function spaces.
This, in particular, applies to functions with high smoothness. (A similar property is now also known for digital nets, see \cite{GSY16, GSY17}.) However, it remains a difficult problem to determine which points from the shrunk lattice belong to the unit cube $[0,1]^d$.

In this paper we 
make a first attempt to construct an universal digital net by establishing 
a polynomial analog of Frolov's cubature formula.
Again polynomial analog means that $\bR$ and $\bZ$ in Frolov's cubature formula 
are replaced by $\Fb((x^{-1}))$ and $\Fb[x]$, respectively, 
where $b$ is a prime and $\Fb$ is the field of $b$ elements.
Further, 
$[0,1]^d$ is replaced by $\Ucube[d]_b$, where $\Ucube[]_b$ consists of all the 
elements of the form $\sum_{i=1}^{\infty} c_i x^{-i}$ with $c_i \in \Fb$.
Thus the integration nodes of our cubature formula are given as follows:
choose a suitable full rank $\Fb[x]$-lattice $\bbX$ in $\Fb((x^{-1}))^d$,
polynomially shrink it, restrict it to $\Ucube[d]_b$,
and map it to the unit cube by replacing $x^{-1}$ by $b^{-1}$, 
see Definition~\ref{def:P} below. Here we mention the papers of Levin~\cite{Lev2010,Lev2017} who studied a similar approach but with different techniques. 

Let $P=P_{\bsf}:=\phi(\bsf^{-1} \bbX)\cap[0,1]^d$ be the point set that is constructed this way, 
see~\eqref{eq:map} and~\eqref{eq:def-P}, 
where $\bsf\in \Lfield^d$ is the shrinking factor, 
and let
\[
Q_P(\psi) \,=\, \frac{1}{|P|}\sum_{\bsx\in P} \psi(\bsx)
\]
be the corresponding QMC rule. 

We will prove that, if the underlying lattice $\bbX$ is {\em admissible} 
(Definition~\ref{def:adm}), then for a sufficiently large shrinking factor, 
we have that $P$ is a $(t,m,d)$-net (Theorem~\ref{thm:t-value-formula}).
Moreover, we show that the sets $P$ can be computed explicitly 
(Theorem~\ref{thm:number-of-points}) 
and provide an explicit construction of an admissible lattice 
that fulfills our assumptions 
(Theorem~\ref{thm:construction}).

The equidistribution property of the point set above implies a bound
on the star-discrepancy, which is a quantitative measure for the irregularity 
of distribution of a point set $P$ in $[0,1]^d$. It is defined as 
\[
D^*(P)=\sup_{J} \left| \frac{|P \cap J|}{|P|}-\lambda_d(J)\right|,
\]
where the supremum is extended over all axes-parallel boxes 
$J \subseteq [0,1)^d$ with one vertex anchored in the origin and where 
$\lambda_d(J)$ denotes the volume of the box $J$. 
Furthermore, $|P \cap J|$ is the number of elements of $P$ that belong to $J$. 

We can deduce from Corollary~\ref{cor:disc} that for a fixed admissible 
lattice $\bbX$ and for all $N$ of the form $b^n$ 
($b$ prime, $n$ large enough) there exists a shrinking factor $\bsf$ such 
that $N=|P_{\bsf}|$ and 
\[
D^*(P_{\bsf}) \;\le\; c\,\frac{(\log N)^{d-1}}{N},
\]
where the constant $c=c(\bbX)$ does not depend on $N$ (or $\bsf$), 
but possibly on $b$ and $d$.
Using the Koksma-Hlawka inequality we likewise obtain 
\[
\left|Q_{P_{\bsf}}(\psi) - \int_{[0,1]^d}\psi(\bsx) \rd\bsx\right| 
\;\le\; c\,\frac{(\log N)^{d-1}}{N} \;{\rm Var}_{HK}(\psi),
\]
where ${\rm Var}_{HK}$ denotes the variation in the sense of Hardy and Krause, see \cite{Niederreiter1992rng}.

There are also other constructions of digital nets that achieve these bounds
including the Sobol$'$ sequence \cite{Sobolcprime1967dpi},
Faure sequence \cite{Faure1982dds},
Niederreiter sequence \cite{Niederreiter1988ldl} 
and Niederreiter-Xing sequence \cite{Niederreiter2001rpo}.

Our construction is now by scaling of a fixed (admissible) digital net with 
infinite digit expansion, which is very much like Frolov's cubature rule. 
We focus on the $(t,m,d)$-net property of our integration nodes, which turn out to be a special case of
digital nets with infinite digit expansion \cite{Goda2016dni}.
We extend the notion of the $(t,m,d)$-net property and duality theory, which relates the $t$-value of a digital net to the minimum Niederreiter-Rosenbloom-Tsfasman (NRT) weight of its dual net.
From this duality theory we obtain a property for a full rank $\Fb[x]$-lattice
such that the integration nodes constructed from the lattice have bounded $t$-values.
As a consequence, we obtain that our cubature rules are QMC rules using a sequence of digital nets with infinite digit expansion with low discrepancy. 

The rest of the paper is organized as follows.
In Section~\ref{sec:Frolov} we recall Frolov's  cubature formula.
In Section~\ref{sec:notation} we introduce notation for our cubature formula,
a polynomial analog of Frolov's, 
and adapt the notation from admissible lattices in $\bR^d$ 
to $\Fb((x^{-1}))^d$.
In Section~\ref{sec:points-in-cube},
we give the number of nodes in the unit cube precisely
and an algorithm to compute them when the shrinking factor is sufficiently large.
In Section~\ref{sec:group} we give preliminary results for the underlying 
structure of the integration nodes.
In Section~\ref{sec:du_th} we establish the corresponding duality theory 
and in Section~\ref{sec:construction} we give an explicit construction of 
a suitable full rank $\Fb[x]$-lattice.


\section{Frolov's cubature formula}\label{sec:Frolov}
Since our construction is inspired by Frolov's cubature formula~\cite{Fr76,Fr80}, 
we first introduce it here. We follow the presentation in 
\cite{Ullrich2016ueb}, see also~\cite{Te93}.

Let $\bbX \subset \bR^d$ be a $d$-dimensional lattice, i.e., there exists an invertible $d \times d$ matrix $T$ over $\bR$ such that
\[
\bbX=T(\bZ^d)=\{T\bsx \ : \  \bsx \in \bZ^d\}.
\]
For every real number $a \neq 0$,
the multiple $a\bbX = \{a \bsx \ : \  \bsx \in \bbX\}$ is again a $d$-dimensional lattice.
A \emph{general lattice rule} is an equally-weighted cubature rule using 
points of a shrunk lattice $a^{-1}\bbX$, with  $a>1$, inside the unit cube 
$[0,1)^d$, i.e.,
\begin{equation}\label{eq:Frolalg}
Q_a(\psi)
= \frac{\det(T)}{a^d} \sum_{\bsx \in \Pcal_{a,d}} \psi(\bsx) 
\end{equation}
with
\[
\Pcal_{a,d}=a^{-1}\bbX \cap [0,1)^d.
\]
It is well-known that $|\Pcal_{a,d}| \sim a^d \det(T^{-1})$, i.e., $|\Pcal_{a,d}|/(a^d \det(T^{-1}))$ tends to 1 as $a$ tends to infinity.

It remains to specify the \emph{generator} $T\in\bR^{d\times d}$ 
of the lattice $\bbX$. 
The choice of Frolov~\cite{Fr76} was as follows:
Define the polynomial $p_d \in \bZ[x]$, $p_d(x)=-1+\prod_{j=1}^d (x-2j+1)$.
Then $p_d$ has $d$ different real roots, say $\xi_1, \ldots, \xi_d \in \bR$.
With these roots he defined the $d \times d$ Vandermonde matrix $B$ by
\[
B=(B_{i,j})_{i,j=1}^d =(\xi_i^{j-1})_{i,j=1}^d \ \ \ \mbox{ and }\ \ \ T = (B^{-1})^\top.\] 

These lattices have many interesting properties. 
In particular, the corresponding cubature rule $Q_a$, after a suitable 
modification~\cite{NUU17}, exhibits 
the optimal order of convergence in many function classes, 
see e.g.~\cite{Du97,Fr76,Te93,Ullrich2017mc,Ullrich2016rfc} 
for results on Sobolev, Besov and Triebel-Lizorkin spaces.
Moreover, the sets $\Pcal_{a,d}$ satisfy certain optimal discrepancy 
bounds~\cite{Fr80,Skriganov1994cud}.

The important property of such lattices can be described best in terms 
of the \emph{dual lattice} $\DinL{\bbX}$ 
of $\bbX$. That is,  
\[
\DinL{\bbX}
:= \{\bsy \in \bR^d \ : \  \bsx \cdot \bsy \in \bZ \text{ for all $\bsx \in \bbX$}\},
\]
where $\bsx\cdot\bsy$ is the usual inner product of $\bsx,\bsy \in \mathbb{R}^d$.
The dual lattice associated with the lattice $\bbX=T(\bZ^d)$
is given by $\DinL{\bbX}= B(\bZ^d)$ with $B = (T^{-1})^\top$.
The lattices from above then satisfy
\[
\prod_{i=1}^d |z_i| \,\geq\, c \,>\,0
\qquad\text{ for all }\quad \bsz=(z_1,\dots,z_d)\in\DinL{\bbX}\setminus\{\bszero\}
\]
and some $c>0$. 
Lattices that satisfy this property are called \emph{admissible lattices}. 
See e.g.~\cite{KOUU17,KOU17,Ullrich2016rfc} for additional properties and other constructions.


\section{Notation and definition of admissible digital nets}\label{sec:notation}

Throughout let $\bR$, $\bZ$, $\bNm$ and $\bN$ be the set of real numbers,
integers, non-positive integers and positive integers, respectively.
Let $b$ be a prime number and let $\Fb$ denote the finite field of order $b$. We identify $\Fb$ with the set $\bZ_b=\{0,1,\ldots,b-1\}$ equipped with arithmetic operations modulo $b$.
Let $\Fb((x^{-1}))$ be the field of Laurent series over $\Fb$. Elements of $\Fb((x^{-1}))$ are formal Laurent series of the form 
\begin{equation}\label{Lseries}
f=\sum_{i=w}^{\infty} c_i x^{-i},
\end{equation}
where $w$ is an arbitrary integer and all $c_i \in \Fb$.
Furthermore, let $\Fb[x]$ be the ring of polynomials over $\Fb$.

We expand the notion of degree from $\Fb[x]$ to $\Fb((x^{-1}))$:  
For $f$ as in \eqref{Lseries} we define $\deg(f) := -w$ if $f\not=0$ and 
$w$ is the least index with $c_w\not=0$. 
For $f=0$, we define $\deg(0) = -\infty$. 
We remark that often the degree $\deg$ is also called the {\it discrete 
exponential valuation} on $\Fb((x^{-1}))$. It is easy to see, that 
for $f,g \in \Fb((x^{-1}))$ we have 
$$\deg(fg)=\deg(f)+\deg(g) \ \ \mbox{ and }\ \ \deg(f+g) \le \deg(f)+\deg(g).$$

For $d \in \bN$ we define the {\it digital unit cube} as $$\Ucube_b := \{ (f_1, \dots, f_d) \in \Lfield^d \ : \  \deg(f_j) < 0 \text{ for all $j$} \}.$$
The field $\Lfield$ is related to $\bR$ through the map $\phi$ defined as
\begin{equation}\label{eq:map}
\phi \colon \Lfield \to \bR, \quad \sum_{i=w}^{\infty} c_i x^{-i} \mapsto \sum_{i=w}^{\infty} c_i b^{-i}.
\end{equation}
We also consider its truncated version $\phi_n$ defined as
\[
\phi_n \colon \Ucube[1]_b \to [0,1), \quad \sum_{i=1}^{\infty} c_i x^{-i} \mapsto \sum_{i=1}^{n} c_i b^{-i}.
\]

Let $\bbX$ be a $d$-dimensional $\Pring$-lattice in $\Lfield^d$.
That is, there exists $T \in {\rm GL}_d\big(\Lfield\big)$,
where ${\rm GL}_d(\Lfield)$ is the set of $d \times d$ invertible matrices over $\Lfield$,
such that
\[
\bbX=\bbX_T:=T(\Pring^d)=\{T\bsg \ : \  \bsg \in \Pring^d\}.
\]
For $\bsf = (f_1, \dots, f_d) \in \Lfield^d$,
the multiple $$\bsf \bbX := \{ (f_1 g_1,\ldots,f_d g_d) \ : \ (g_1,\ldots,g_d) \in \bbX\}$$
is again a $d$-dimensional $\Pring$-lattice in $\Lfield^d$.
The generating matrix of $\bsf \bbX$ is given by $\Diagm{\bsf} T$,
where $\Diagm{\bsf} = {\rm diag}(f_1, \dots, f_d)$ is the $d \times d$ diagonal matrix
whose diagonal entries are $f_1, \dots, f_d$.

\begin{definition}\label{def:P}
For a $d$-dimensional $\Pring$-lattice $\bbX=\bbX_T$ in $\Lfield^d$ and for $\bsf = (f_1, \dots, f_d)\in \Lfield^d$ with $\deg(f_j) \geq 0$
we denote by $\Ptset$ all points of a polynomially shrunken lattice $\bsf^{-1}\bbX$ inside the digital unit cube $\Ucube_b$, i.e.,
\begin{equation}\label{eq:def-P}
\Ptset := \bsf^{-1} \bbX \cap \Ucube_b,
\end{equation}
where $\bsf^{-1}=(f_1^{-1},\ldots,f_d^{-1})$, and $f^{-1}$ denotes the inverse of $f$ in $\Lfield$.
In this context we call $\bsf$ the {\it shrinking factor}.
\end{definition}

The elements from $\Ptset$ correspond to nodes in $[0,1]^d$ through the map $\phi$.
Note that if $\bsf  = (x^a, \dots, x^a)$ with $a \in \bZ$
then $\phi(\bsf^{-1} \bbX)$ is a geometrical shrinking of $\phi(\bbX)$ in 
$\bR^d$ with shrinking factor $b^{-a}$.

Analogous to \eqref{eq:Frolalg}, we consider a cubature rule with equal weights as
\begin{equation}\label{eq:ouralg}
Q_{\bsf}(\psi)= \frac{b^{\deg(\det(T))}}{b^{\sum_{j=1}^d \deg(f_j)}} \sum_{\bsx \in \phi(\Ptset)} \psi(\bsx).
\end{equation}
We will show in Theorem~\ref{thm:t-value-formula} that $b^{\sum_{j=1}^d \deg(f_j) - \deg(\det(T))} = |\phi(\Ptset)|$ and hence this rule is a QMC rule.

In the following we investigate the $(t,m,d)$-net property for the set 
$\phi(\Ptset)$ for a properly chosen matrix $T$.
The point set $\Ptset$ is closed under summation in $\Lfield$.
Since an element in $\Lfield$ can have an infinite digit expansion, the point set
$\phi(\Ptset)$ is not always closed under digit-wise addition in $[0,1)^d$.
This motivates us to define the notion of digital nets in $\Ucube_b$.
Note that digital nets with infinite digit expansion were already 
considered in~\cite{Goda2016dni}.

First we recall the definition of $(t,m,d)$-nets in base $b$ according to Niederreiter~\cite{Niederreiter1992rng}.

\begin{definition}[Niederreiter]\label{def:tms-net-cube}
Let $b \geq 2$, $m \geq 1$, $0 \leq t \leq m$, and $d \geq 1$ be integers.
A point set $P = \{\bsx_0, \dots , \bsx_{b^m - 1}\}$ in $[0,1)^d$
is called a $(t,m,d)$-net in base $b$
if for all nonnegative integers $l_1, \dots , l_d$
with $l_1 + \dots + l_d = m - t$
the elementary intervals
\[
\prod_{j=1}^{d}{\left[\frac{a_j}{b^{l_j}} , \frac{a_j +1}{b^{l_j}}\right)}
\]
contain exactly $b^t$ points of $P$ for all choices of
$0 \leq a_j < b^{l_j}$ $(a_j \in \bZ)$ for $1 \leq j \leq d$.
\end{definition}

\medskip

We now define a $(t,m,d)$-net in $\Ucube_b$.
\begin{definition}\label{def:tms-net-ring}
Let $b \geq 2$ be a prime and $m \geq 1$, $0 \leq t \leq m$, and $d \geq 1$ be integers.
A point set 
$\Pcal = \{\bsg_0, \dots , \bsg_{b^m - 1}\}$ in $\Ucube_b$
is called a $(t,m,d)$-net in $\Ucube_b$
if for all nonnegative integers $l_1, \dots , l_d$
with $l_1 + \dots + l_d = m - t$
the ``elementary intervals''
\[
\big\{\bsg=(g_1, \dots, g_d)\in\Ucube_b \ : \  \deg (g_j - h_j) < -l_j  
\  \text{for all $1 \leq j \leq d$} \big\}
\]
contain exactly $b^t$ points of $\Pcal$ for all choices of
$h_j \in \TrBox[1]{b,l_j}$ for $1 \leq j \leq d$, where
\[
\TrBox[1]{b,l_j} = \{c_1 x^{-1} + c_2 x^{-2} + \cdots + c_{l_j} x^{-l_j} \ : \  c_1, c_2, \ldots, c_{l_j} \in \mathbb{F}_b \}.
\]
\end{definition}

The following lemma provides a relation between the above two definitions. We omit the easy proof.

\begin{lemma}\label{lem:def_net}
A set $\mathcal{P}$ is a $(t,m,d)$-net in $\Ucube_b$
if and only if $\phi_m(\mathcal{P})$ is a $(t,m,d)$-net in base $b$.
\end{lemma}

\begin{remark}\label{rem:dnet-DM}\rm
By Lemma~\ref{lem:def_net}, Definition~\ref{def:tms-net-ring} can be viewed as an extension of \cite[Definition~2.4]{Dick2013fcw}, where $(t,m,d)$-nets in $\Fb^{d \times n}$ are introduced for $n \geq m$.
If $\Pcal \subset \TrBox{b,n} \ (\simeq \Fb^{d \times n})$ then \cite[Definition~2.4]{Dick2013fcw} and Definition~\ref{def:tms-net-ring} coincide.
\end{remark}

Analogous to the notion of admissible lattice in $\bR^d$, 
see Section~\ref{sec:Frolov}, we define the notion of 
\emph{admissible $\Pring$-lattices in $\Lfield^d$}.

For a $d$-dimensional $\Pring$-lattice $\bbX$ in $\Lfield^d$ the dual lattice $\DinL{\bbX}$ is defined as
\[
\DinL{\bbX}
:= \Big\{\bsg \in \Lfield^d \ : \ \bsg \cdot \bsh \in \Pring \text{ for all $\bsh \in \bbX$}\Big\},
\]
where $\bsg \cdot \bsh$ is the usual inner product of $\bsg, \bsh \in \Lfield^d$,
given by $\bsg \cdot \bsh = \sum_{j=1}^d g_j h_j$.
The dual lattice associated with the lattice $\bbX=T(\Pring^d)$
is given by $\DinL{\bbX}= B(\Pring^d)$ with $B = (T^{-1})^\top$.
We define
\begin{equation}
\adnum(\bbX)
:= \inf_{(g_1, \dots, g_d) \in \DinL{\bbX}\setminus \{\bszero\}}  \sum_{j=1}^d \deg(g_j).
\end{equation}

\begin{definition}\label{def:adm}
We say that $\bbX$ is admissible if $\adnum(\bbX) > -\infty$.
\end{definition}

\begin{theorem}\label{thm:t-value-formula}
Let $\bbX$ be a $d$-dimensional admissible lattice with the generating 
matrix 
$T\in {\rm GL}_d\big(\Lfield\big)$.
If the shrinking factor $\bsf$ is ``sufficiently large'' (more explicit, if it satisfies condition \eqref{eq:condition-shrinking-factor} below),
then $\Ptset$ is a $(t,m,d)$-net in $\Ucube_b$ with
$$t = - \deg(\det T) - \adnum(\bbX) - d + 1$$
and
$$m = - \deg(\det T) + \sum_{j=1}^{d} \deg (f_j).$$
In particular, the $t$-values of point sets made from a sufficiently shrunk 
admissible lattice are uniformly bounded  and the corresponding cubature rule $Q_{\bsf}$ from~\eqref{eq:ouralg} 
is a QMC rule.
\end{theorem}

A similar result has been obtained by Levin \cite[Theorem~3.2]{Lev2010}, who considered $(d+1)$-dimensional lattices in  $\bF_b((x^{-1}))^{d+1}$ to construct uniformly distributed
sequences in $[0,1)^d$.

The proof of Theorem~\ref{thm:t-value-formula} is split into two parts. 
In Section~\ref{sec:points-in-cube} we show the result for $m$,
i.e.~the logarithm of the number of points, and we provide an  algorithm for finding the elements of $\Ptset$ explicitly. 
In Section~\ref{sec:du_th} we will prove the result for the $t$-value. 
This proof will rely on duality theory for these nets from Section~\ref{sec:group}. 
In Section~\ref{sec:construction} we will provide an explicit construction of a matrix $T$.

We close this section with a remark and a corollary to Theorem~\ref{thm:t-value-formula}.

\begin{remark}\rm
Some authors consider the {\it strength} of a $(t,m,d)$-net, given by $\rho=m-t$, as quality parameter, see \cite[Remark~4.53]{Dick2010dna}. According to Theorem~\ref{thm:t-value-formula} the strength of the $(t,m,d)$-net $\Ptset$ is $$\rho=\sum_{j=1}^{d} \deg (f_j) + \adnum(\bbX)+d-1.$$
\end{remark}
\medskip


Together with \cite[Theorem 4.10]{Niederreiter1992rng} (which we use in the form as presented in \cite[Theorem 5.10]{Dick2010dna}) we obtain the following bound on the star-discrepancy of $\phi(\Ptset)$.

\begin{corollary}\label{cor:disc}
Let $\Ptset$ be the net from Theorem~\ref{thm:t-value-formula} and let $$N_{\boldsymbol{f}}=|\Ptset|=b^{- \deg(\det T) + \sum_{j=1}^{d} \deg (f_j)}.$$ 
Then the star-discrepancy of $\phi(\Ptset)$ satisfies 
\[
D^*(\phi(\Ptset)) \ll_{b,d} b^t 
\frac{(\log_b N_{\boldsymbol{f}}-t)^{d-1}}{N_{\boldsymbol{f}}} 
= \frac{\left(\sum_{j=1}^{d} \deg (f_j) + \adnum(\bbX)+d-t-1\right)^{d-1} }{b^{\sum_{j=1}^{d} \deg (f_j) + \adnum(\bbX)+d-1}}.
\]
\end{corollary}


\section{Determination of the points of $\Ptset$}\label{sec:points-in-cube}

The purpose of this section is to prove that $\Ptset$ consists of $b^{- \deg(\det T) + \sum_{j=1}^{d} \deg (f_j)}$ elements. Furthermore, we provide an algorithm for finding the elements of $\Ptset$ explicitly for sufficiently shrunken lattices. Mahler \cite[Section~8]{mahler1941} gave such a result when the shrinking factors are given by $f_j = x^r$ for all $j$, for sufficiently large $r \in \bN$, by establishing an analogue to Minkowski's geometry of numbers in a field of formal power series. In this paper, we give an elementary proof for sufficiently large but arbitrary shrinking factors. The main results of this section are summarized in Theorem~\ref{thm:number-of-points}.

We begin with the generating matrix $T \in \mathrm{GL}_d(\Lfield)$ of the lattice. Let $$\Lring=\{f \in \Lfield \ : \ \deg(f)\le 0\}$$ be the ring of formal power series over $\Fb$. Elements of $\Lring$ are of the form $\sum_{i=0}^{\infty} c_i x^{-i}$. 


Using an $LQ$ decomposition (a variant of the standard $QR$ decomposition), we can decompose $T$ as $$T = L'Q,$$
where $L' = (l'_{ij})_{i,j=1}^d$ is lower triangular matrix and where the matrix $Q$ satisfies
\begin{equation}\label{eq:cond-Q}
\text{$Q = (q_{ij})_{i,j=1}^d \in \Lring^{d \times d}$ \quad and \quad $\det(Q) = 1$.} 
\end{equation}

We now consider $\Ptset$ whose shrinking factor $\bsf=(f_1,\ldots,f_d)$ satisfies 
\begin{equation}\label{eq:condition-shrinking-factor}
\deg(f_j) \geq \max_{1 \le i \le j} \deg(l'_{ji}) \ \ \ \mbox{for all $1 \le j \le d$}.
\end{equation}
Recall that
\[
\Ptset
= \Diagm{\bsf}^{-1} \bbX \cap \Ucube_b
= \Diagm{\bsf}^{-1}L'Q(\Pring^d) \cap \Ucube_b
= LQ(\Pring^d) \cap \Ucube_b,
\]
where we put $L = \Diagm{\bsf}^{-1}L'$. Then $L=(l_{ij})_{i,j=1}^d$ satisfies 
\begin{equation}\label{eq:cond-L}
\text{$L$ is lower triangular and
$\deg(l_{ij}) \leq 0$ for all $1 \leq i,j \leq d$.}
\end{equation}

For $g = \sum_{i=w}^{\infty} c_i x^{-i} \in \Lfield$ define the polynomial part  $\Ppart{\cdot} : \Lfield \to \Pring$ as 
\begin{equation}\label{eq:def_poly}
\Ppart{g} := 
\begin{cases}
\sum_{i=w}^{0} c_i x^{-i} & \text{if $w \leq 0$,}\\
0 & \text{otherwise.}
\end{cases}
\end{equation}
For $\bsg = (g_1, \dots, g_d) \in \Lfield^d$, we define $\Ppart{\bsg} = (\Ppart{g_1}, \dots, \Ppart{g_d}) \in \Fb[x]^d$.

\begin{lemma}\label{lem:Q-property}
For a matrix $Q$ satisfying \eqref{eq:cond-Q}, the following holds true:
\begin{enuroman}
\item \label{lem:Q-property-1}
For $e \in \bZ$ we define $x^e\Ucube_b := \{x^e \bsg \ : \ \bsg \in \Ucube_b\}$. Consider the linear map 
\[
Q : \Lfield^d \to \Lfield^d; \quad \bsg \mapsto Q\bsg.
\]
Then the map $Q\big|_{x^e\Ucube_b}$ with restricted domain $x^e\Ucube_b$ is a bijection from $x^e\Ucube_b$ to $x^e\Ucube_b$.
\item \label{lem:Q-property-2}
Let $\Ucube_b+\bsg := \{\bsh \in \Lfield^d \ : \ \bsh -\bsg \in \Ucube_b \}$.
Then $Q\big|_{\Ucube_b+\bsg}$ is a bijection from $\Ucube_b+\bsg$ to $\Ucube_b+Q\bsg$.
\item \label{lem:Q-property-3}
The map $\PtoPmap{Q} : \Pring^d \to \Pring^d$ given by $\PtoPmap{Q}(\bsg) := \Ppart{Q\bsg}$ is bijective and its inverse map is $\PtoPmap{Q^{-1}}$.
\end{enuroman}
\end{lemma}

\begin{proof}
First we prove \eqref{lem:Q-property-1}.
Since $Q \circ Q^{-1} = Q^{-1} \circ Q = \mathrm{id}_{\Lfield^d}$,
it suffices to show
$Q(x^e_b) \subset x^e\Ucube_b$ and $Q^{-1}(x^e\Ucube_b) \subset x^e\Ucube_b$.
Let $\bsg = (g_1, \dots, g_d) \in x^e\Ucube_b$.
Then we have $Q\bsg = (\sum_{j=1}^d q_{ij} g_j)_{i=1}^d$.
For each component $\sum_{j=1}^d q_{ij} g_j$,
we have
\[
\deg \left(\sum_{j=1}^d q_{ij} g_j \right)
\leq \max_{j=1, \dots, d} \deg(q_{ij} g_j)
\leq \max_{j=1, \dots, d} (\deg (q_{ij}) + \deg (g_j))
\leq e,
\]
where the last inequality follows from \eqref{eq:cond-Q} and $\deg(g_j) \leq e$ for all $j$.
This shows $Q(x^e\Ucube_b) \subset x^e\Ucube_b$.
Since $Q^{-1}$ also satisfies \eqref{eq:cond-Q}, we can replace $Q$ by $Q^{-1}$ in the above argument
and we have $Q^{-1}(x^e\Ucube_b) \subseteq x^e\Ucube_b$.
Thus, we have proved \eqref{lem:Q-property-1}.

Using \eqref{lem:Q-property-1} for $e=0$, we can see that \eqref{lem:Q-property-2} holds.

We now prove \eqref{lem:Q-property-3}. Let $\bsg \in \Pring^d$.
Since $\PtoPmap{Q}(\bsg) \in \Ucube_b+Q\bsg$,
it follows from \eqref{lem:Q-property-2} that $Q^{-1}(\PtoPmap{Q}(\bsg)) \in \Ucube_b+\bsg$.
Thus $\PtoPmap{Q^{-1}}(\PtoPmap{Q}(\bsg)) = \bsg$.
In the same manner we have $\PtoPmap{Q}(\PtoPmap{Q^{-1}}(\bsg)) = \bsg$.
This proves that the maps $\PtoPmap{Q}$ and $\PtoPmap{Q^{-1}}$ are the inverse maps of each other.
\end{proof}

\begin{lemma}\label{lem:L-property}
Let $L$ be a lower triangular matrix satisfying \eqref{eq:cond-L}.
For a set $V \subset \Lfield^d$ we define $L(V) := \{L\bsg \ : \ \bsg \in V\}$.
Then the following holds true:
\begin{enuroman}
\item \label{lem:L-property-1}
We have $L(\Ucube_b) \subseteq \Ucube_b$.
\item \label{lem:L-property-2}
Let $\bsg, \bsh \in \Lfield^d$ with $\bsg - \bsh \in \Ucube_b$.
Then $$L\bsg \in \Ucube_b \ \Leftrightarrow \ L\bsh \in \Ucube_b.$$
\item \label{lem:L-property-3}
Let $\mathbb{S} := \{\bsg \in \Pring^d \ : \  L\bsg \in \Ucube_b \}$.
Then we have $|\mathbb{S}| = b^{-\sum_{j=1}^d \deg(l_{jj})}$.
\end{enuroman}
\end{lemma}

\begin{proof}
Since the proof of \eqref{lem:L-property-1} works in the same manner as that of Lemma~\ref{lem:Q-property} \eqref{lem:Q-property-1}, we omit it.

We now show \eqref{lem:L-property-2}: According to \eqref{lem:L-property-1} we have $L(\bsh - \bsg) \in \Ucube_b$. Now assume that $L\bsg \in \Ucube_b$. Then, by linearity, we have 
$$L\bsh = L(\bsh - \bsg) + L\bsg \in \Ucube_b.$$
The converse holds true in the same way.

We now prove \eqref{lem:L-property-3} by induction on $d$.
First we assume $d=1$.
For $g \in \Pring$, $g \in \mathbb{S}$ is equivalent to $l_{11} g \in \Ucube[1]_b$,
which holds if and only if $\deg{g} < - \deg(l_{11})$.
Thus $|\mathbb{S}| =  b^{- \deg(l_{11})}$ and \eqref{lem:L-property-3} holds for $d=1$.

We now assume \eqref{lem:L-property-3} holds for $d-1$ and prove the result for $d$.
We remark that the following induction gives an algorithm to determine $\mathbb{S}$.
Let $L^{(d-1)}$ be the left-upper $(d-1) \times (d-1)$ sub-matrix of $L$
and $\mathbb{S}^{(d-1)} := \{\bsg^{(d-1)} \in \Pring^{d-1} \ : \ L^{(d-1)} \bsg^{(d-1)} \in \Ucube[d-1]_b \}$,
where $|\mathbb{S}^{(d-1)}| = b^{-\sum_{j=1}^{d-1} \deg(l_{jj})}$ by induction assumption.
For $\bsg = (g_1, \dots, g_d) \in \Pring^d$,
$\bsg \in \mathbb{S}$ is equivalent to $(g_1, \dots, g_{d-1}) \in \mathbb{S}^{(d-1)}$
and $\sum_{i=1}^{d} l_{di} g_i \in \Ucube[1]_b$.
The latter is equivalent to $$\deg\left(g_{d} + l_{dd}^{-1} \sum_{i=1}^{d-1} l_{di} g_i\right) < - \deg(l_{dd}),$$
which holds if and only if
\begin{equation}\label{eq:S-equiv-cond}
g_{d} = \Ppart{- l_{dd}^{-1} \sum_{i=1}^{d-1} l_{di} g_i} + h
\quad \text{with $h \in \Pring$ such that $\deg{h} < - \deg(l_{dd})$}.
\end{equation}
Thus
\[
\mathbb{S} = \{(g_1,\ldots,g_{d-1}, g_d) \in \Pring^d \ : \ \bsg^{(d-1)}=(g_1,\ldots,g_{d-1}) \in \mathbb{S}^{(d-1)} \text{ and $g_d$ satisfies \eqref{eq:S-equiv-cond}}\}.
\]
The number of $g_{d}$ which satisfy \eqref{eq:S-equiv-cond} is $b^{- \deg(l_{dd})}$
for fixed $\bsg^{(d-1)} \in \mathbb{S}^{(d-1)}$,
since we have $\deg(l_{dd}) \leq 0$ from \eqref{eq:cond-L}.
Thus
we have $|\mathbb{S}| = b^{- \deg(l_{dd})}|\mathbb{S}^{(d-1)}| =b^{-\sum_{j=1}^{d} \deg(l_{jj})}$,
which proves \eqref{lem:L-property-3} for $d$.
\end{proof}

We obtain the following explicit algorithm to compute the set $\mathbb{S}$.
\begin{algorithm}\label{alg:giveS}
Given $T$ and shrinking factor $\bsf$.

\begin{enumerate}
\item Compute a $LQ$ decomposition of the matrix $T$ such that $T = L'Q$, where $L'$ is a lower triangular matrix and $Q$ satisfies \eqref{eq:cond-Q}.
\item Set $D_{\bsf} = {\rm diag}(f_1, \ldots, f_d)$ and compute its inverse $D^{-1}_{\bsf}$.
\item Compute $L = D^{-1}_{\bsf} L'$. Let $L = (l_{ij})$.
\item Compute the set $\mathbb{S}^{(1)} = \{g \in \Pring\ : \ \deg{g} < - \deg(l_{11})\}$.
\item For $j = 2, 3, \ldots, d$ do:

Compute the set $\mathbb{S}^{(j)} = \{(g_1,\ldots,g_{j-1}, g_j) \in \Pring^j \ : \ \bsg^{(j-1)}=(g_1,\ldots,g_{j-1}) \in \mathbb{S}^{(j-1)} \text{ and $g_j$ satisfies \eqref{eq:S-equiv-cond}} \}$.
\end{enumerate}
\end{algorithm}

We return to the analysis of $LQ(\Pring^d) \cap \Ucube_b$.
For $\bsg \in \Pring^d$
we have the equivalence
\[ LQ\bsg \in \Ucube_b
\Leftrightarrow L\Ppart{Q\bsg} \in \Ucube_b
\Leftrightarrow \Ppart{Q\bsg} \in \mathbb{S}
\Leftrightarrow \bsg \in \PtoPmap{Q^{-1}}(\mathbb{S}),
\]
where the first, second and the last equivalences follow from
Lemma~\ref{lem:L-property} \eqref{lem:L-property-2}, the definition of $\mathbb{S}$ and
Lemma~\ref{lem:Q-property} \eqref{lem:Q-property-3}, respectively.
Therefore we have
\begin{equation}\label{eq:point-number}
|LQ(\Pring^d) \cap \Ucube_b| = |\PtoPmap{Q^{-1}}(\mathbb{S})| = |\mathbb{S}| = b^{-\sum_{j=1}^{d} \deg(l_{jj})}.
\end{equation}
From $T = L'Q = \Diagm{\bsf}LQ$ and $\det Q=1$ it follows that 
\[
\det T = \det \Diagm{\bsf} \det L \det Q = \prod_{j=1}^d f_{j}  l_{jj} 
\]
and thus
\begin{equation}\label{eq:detT}
\deg (\det T) = \sum_{j=1}^d (\deg (f_{j}) + \deg (l_{jj})).
\end{equation}
Combining \eqref{eq:point-number} and \eqref{eq:detT} we obtain 
\[
|LQ(\Pring^d) \cap \Ucube_b| = b^{- \deg(\det T) + \sum_{j=1}^{d} \deg (f_j)}.
\]
Thus, we have proved the following.

\begin{theorem}\label{thm:number-of-points}
Let $T \in \mathrm{GL}_d(\Lfield)$ be the generating matrix of the lattice.
Consider a decomposition $T = L'Q$, where $L' = (l'_{ij})_{i,j=1}^d$ is lower triangular and $Q = (q_{ij})_{i,j=1}^d$ satisfies \eqref{eq:cond-Q}.
Assume that $\bsf = (f_1, \dots, f_d) \in \Pring^d$ is the shrinking factor satisfying \eqref{eq:condition-shrinking-factor}.
Then we have
\[
|\Ptset| = b^{- \deg(\det T) + \sum_{j=1}^{d} \deg (f_j)}.
\]
Further $\Ptset$ is given by
\[
\Ptset = \{ LQ\Ppart{Q^{-1}\bsg} \ : \  \bsg \in \mathbb{S}\},
\]
where $\mathbb{S}$ is defined as in Lemma~\ref{lem:L-property} \eqref{lem:L-property-3}
and can be explicitly computed inductively as in Algorithm~\ref{alg:giveS}. 
\end{theorem}

\section{Group theoretic properties}\label{sec:group}
In this section, based on \cite{Forney2004dgc}
we introduce the underlying group structure and the character groups of $\Lfield^d$ and its subgroups.
See also \cite{Pontryagin1966tg} for Pontryagin Duality and \cite{Fine1950gwf} for the case $b=2$.

\subsection{Underlying group structure}
Here we focus on the underlying (topological) group structure of the point sets under consideration.
First we consider the one-dimensional case. The field $\Lfield$ is a topological field, where the local basis of zero is given by
the set consisting of $\{g \in \Lfield \ : \ \deg(g) < w\} $ for all $w \in \bZ$.
The field $\Lfield$ with this topology is identified with another topological group as follows
(for the case $b=2$, see \cite{Fine1950gwf}):

We consider the direct product of infinitely many copies of $\Fb$, 
$$\Fb^{\bN}= \Fb \times \Fb \times \ldots=\bigotimes_{i \in \bN} \Fb$$ equipped with the product topology, and the direct sum of $\Fb$
\[
\bigoplus_{i \in \bNm} \Fb := \left\{ (c_i)_{i \in \bNm} \in \bigotimes_{i \in \bNm} \Fb
\ : \  \text{$c_i \neq 0$ for only finitely many $i$}\right\},
\]
equipped with the discrete topology.
Then the underlying group structure of $\Lfield$ is identified with
\[
\Lfield \simeq \bigoplus_{i \in \bNm} \Fb \times \Fb^{\bN}
\]
with the product topology, which is indexed by $i \in \bZ$,
by $\sum_{i=w}^{\infty} c_i x^{-i} \mapsto (c_i)_{i \in \bZ}$.
Under this identification, $\Pring$ and $\Ucube[1]_b$, which are subsets of $\Lfield$, are identified with
$\bigoplus_{i \in \bNm} \Fb$ and $\Fb^{\bN}$, respectively.

In this section, we further consider
\[
\TrBox[1]{b,n} := \{c_1 x^{-1} + \cdots + c_n x^{-n} \ : \ c_1, \dots, c_n \in \Fb\},
\]
which is identified with $\Fb^n$,
and the truncation map
\[
\trunc_n \colon \Ucube[1]_b \to \TrBox[1]{b,n}; \quad
\sum_{i=1}^\infty c_i x^{-i} \mapsto \sum_{i=1}^n c_i x^{-i}.
\]

In the $d$-dimensional case, the truncation map is applied componentwise and
the underlying group structure of
$\Lfield^d$, $\Pring^d$, $\Ucube_b$ and $\TrBox{b,n}$
are identified with
\begin{align}
\begin{aligned}\label{eq:identification}
\Lfield^d
&\simeq \left(\bigoplus_{i \in \bNm} \Fb \times \Fb^{\bN} \right)^d,
&\Pring^d
&\simeq \left(\bigoplus_{i \in \bNm} \Fb\right)^d,\\
\Ucube_b
&\simeq \left(\Fb^{\bN}\right)^d,
&\TrBox{b,n}
&\simeq \left(\Fb^n\right)^d.
\end{aligned}
\end{align}

\begin{remark}\rm 
Since a $d$-dimensional lattice restricted to $\Ucube_b \simeq \left(\Fb^{\bN}\right)^d$ is a linear subspace as well as a finite subgroup,
it can be regarded as a digital net with infinite digit expansion as introduced in \cite{Goda2016dni}.
\end{remark}

\subsection{Character group and orthogonal space}
Let $G$ be a topological group.
The character group of $G$, denoted by $\widehat{G}$ or likewise by $\Dgp{G}$,
is the set of all continuous maps from $G$ to the circle group
$\mathbb{T} := \{z \in \bC \ : \  |z| = 1\}$.
We consider a character group of $\Lfield^d$, $\Pring^d$, $\Ucube_b$ and $\TrBox{b,n}$.
Pontryagin duality theory \cite{Pontryagin1966tg} for locally compact groups implies that we have character groups as 
\begin{align*}
\Dgp{(\Lfield^d)}
&\simeq \left(\prod_{i \in \bNm} \widehat{\mathbb{F}}_b \times \bigoplus_{i \in \bN} \widehat{\mathbb{F}}_b \right)^d,
&\Dgp{(\Pring^d)}
&\simeq \left(\bigotimes_{i \in \bNm}  \widehat{\mathbb{F}}_b\right)^d,\\
\Dgp{(\Ucube_b)}
&\simeq \left(\bigoplus_{i \in \bN}  \widehat{\mathbb{F}}_b\right)^d,
&\Dgp{(\TrBox{b,n})}
&\simeq \left(\widehat{\mathbb{F}}_b^n\right)^d.
\end{align*}
We note that $ \widehat{\mathbb{F}}_b$ is (non-canonically) isomorphic to $\Fb$.

We give a more explicit representation of the character groups
$\Dgp{(\Lfield^d)}$, $\Dgp{(\Ucube_b)}$ and $\Dgp{(\TrBox{b,n})}$.
We define a map $\Res : \Lfield \to \Fb$ as
\[
\Res(f) := c_{1} \qquad \text{whenever $f = \sum_{i=w}^{\infty} c_i x^{-i}$,}
\]
where for $w > 1$ we set $\Res(f) = 0$. Then, for fixed $g \in \Lfield$, we define a character $W_{g} \colon \Lfield \to \mathbb{T}$ as
\[
W_{g}(f) = \omegab^{\Res(fg)},
\]
where $\omega_b=\exp(2 \pi \icomp/b)$. It is known that every character on $\Lfield$ is of the form $W_{g}$ for some $g$,
and thus $\Lfield$ is self dual with respect to an isomorphism as a topological group
\begin{align}\label{eq:isom-K-K*}
\Lfield \to \Dgp{\Lfield}, \quad g \mapsto W_{g}.
\end{align}
see \cite[Section~2]{Fine1950gwf} for the case $b=2$.
Under this identification, we have 
$$\Dgp{(\Ucube[1]_b)} = \{W_{g} \ : \ g \in \Pring \} \ \mbox{ and }\ \Dgp{(\Ucube[1]_{b,n})}  = \{W_{g} \ : \  g \in \mathbb{M}_{b,n}\},$$
where $\mathbb{M}_{b,n} := \{g \in \Pring \ : \ \deg(g) < n\}$.

In the $d$-dimensional case, we define a group homomorphism
$W_{\bsg} \colon \Lfield^d \to \mathbb{T}$ 
as
\[
W_{\bsg}(\bsf) = \omegab^{\Res(\bsf \cdot \bsg)}.
\]
Then the characters of $\Lfield^d$, $\Pring^d$, $\Ucube_b$ and $\TrBox{b,n}$ are given by isomorphisms
\begin{equation}\label{eq:isom-characters}
\Lfield^d \simeq \Dgp{(\Lfield^d)},
\quad \Pring^d \simeq \Dgp{(\Ucube_b)},
\quad \Ucube_b \simeq \Dgp{(\Pring^d)},
\quad \mathbb{M}_{b,n}^d \simeq \Dgp{(\TrBox[d]{b,n})}.
\end{equation}

Hereafter we consider the characters under the identification $\bsg \mapsto W_{\bsg}$.

For a closed subgroup $\Pcal \subseteq G$, the orthogonal space of $\Pcal$ in $\widehat{G}$, denoted by $\DinL{\Pcal}$, is defined as
\[
\DinL{\Pcal} = \{\psi \in \widehat{G} \ : \  \psi(p) = 1 \text{ for all $p \in \Pcal$}\}.
\]
In order to distinguish the ambient spaces in which we take the orthogonal complements,
we denote the orthogonal space of $\Pcal$ in $\Lfield^d$, $\Ucube_b$ and $\TrBox{b,n}$
by $\DinL{\Pcal}$,  $\Dnet{\Pcal}$, and $\Dnet[n]{\Pcal}$, respectively,
which are regarded as subsets in 
$\Lfield^d$,  $\Pring^d$, and $\mathbb{M}_{b,n}^d$, respectively, under \eqref{eq:isom-characters}.
Hence if $\Pcal \subseteq \Ucube_b$ then $\Dnet{\Pcal} = \DinL{\Pcal} \cap \Pring^d$
and if $\Pcal \subseteq \TrBox{b,n}$ then $\Dnet[n]{\Pcal} = \Dnet{\Pcal} \cap \mathbb{M}_{b,n}^d$.

The following results give a relation between orthogonal spaces in different groups.

\begin{lemma}\label{lem:dnet}
We have the following.
\begin{enuroman}
\item \label{eq:dnet-i}
Let $\Pcal \subseteq \Lfield^d$ be a closed subgroup.
Then
\[
\Dnet{\Pcal \cap \Ucube_b} = \{\Ppart{\bsf} \ : \  \bsf \in \DinL{\Pcal} \},
\]
where $\Ppart{\bsf}$ is defined as in~\eqref{eq:def_poly}.
\item \label{eq:dnet-ii}
Let $\Pcal \subseteq \Ucube_b$ be a finite subgroup and $\trunc_n(\Pcal) := \{ \trunc_n(\bsg) \ : \  \bsg \in \Pcal \}$.
Then
\[
\Dnet[n]{\trunc_n(\Pcal)} = \Dnet{\Pcal} \cap \mathbb{M}_{b,n}^d. 
\]
\end{enuroman}
\end{lemma}

\begin{proof}
\eqref{eq:dnet-i}:
It follows from \cite[Theorem~4.3]{Forney2004dgc}
that
\begin{equation*}\label{eq:dlat-of-net}
\DinL{(\Pcal \cap \Ucube_b)}
= \{\Ppart{\bsf} + \bsg \ : \ \bsf \in \DinL{\Pcal},\ \bsg \in \Ucube_b \}. 
\end{equation*}
Taking the intersection of this and $\Pring^d$, we obtain the desired result.

\eqref{eq:dnet-ii}:
The result directly follows from \cite[Theorem~4.3]{Forney2004dgc}.
\end{proof}

We note that a $d$-dimensional lattice in $\Lfield^d$ is a closed subgroup and
its restriction to $\Ucube_b$ is finite and closed in $\Ucube_b$.
Note also that the dual lattice $\Dlat{\bbX}$ of a lattice $\bbX$ 
coincides with its orthogonal space $\DinL{\bbX}$,
which is identified with
$$\{\bsg \in \Lfield^d \ : \  \Res(\bsg \cdot \bsh)=0  \text{ for all $\bsh \in \bbX$}\}.$$
Indeed, $\Dlat{\bbX} \subset \DinL{\bbX}$ holds by definition.
We now show the converse inclusion. Let $\bsg \in \DinL{\bbX}$, $\bsh \in \bbX$ and let $\bsg \cdot \bsh = \sum_{i=w}^{\infty} c_i x^{-i}$.
Since $\bsg \in \DinL{\bbX}$ and $x^a \bsh \in \bbX$ for all $a \in \bN$,
we have $c_{a+1} = \Res(\bsg \cdot x^a \bsh) = 0$ for all $a \in \bN_0$.
This implies $\bsg \cdot \bsh \in \Pring$ and thus $\bsg \in \Dlat{\bbX}$.

\section{Duality theory}\label{sec:du_th}

In this section we finish the proof of Theorem~\ref{thm:t-value-formula}. To this end we use duality theory for nets which was first studied and used in  \cite{Niederreiter2001ddn}. Duality theory relates the $(t,m,d)$-net property of digital nets and the NRT weight of its orthogonal space. We define the NRT weight on $\Pring^d$. 
\begin{definition}\label{def:NRT-weight}
For $g \in \Pring$, the NRT weight of $g$ is defined as
\[
\mu(g) =\left\{ 
\begin{array}{ll}
 1 + \deg g & \mbox{ if $g \neq 0$},\\
 0 & \mbox{ if $g=0$.}
\end{array}\right.
\]
For $\bsg = (g_1, \dots, g_d) \in \Pring^d$, the NRT weight of $\bsg$ is defined as
\[
\mu(\bsg) := \sum_{j=1}^d \mu(g_j).
\]
\end{definition}

We define the minimum NRT weights of point sets $\Pcal$ in $\Ucube_b$ as
\[
\minNRT{\Pcal}
:= \inf\{\mu(\bsg) \ : \ \bsg \in \Dnet{\Pcal}\setminus \{\bszero\}\}.
\]
\[
\minNRT[n]{\Pcal}
:= \min\{\mu(\bsg) \ : \  \bsg \in \Dnet[n]{\Pcal}\setminus \{\bszero\}\}
\]

The following result is a direct consequence of Remark~\ref{rem:dnet-DM} and \cite[Theorem~2.8]{Dick2013fcw}.

\begin{lemma}\label{lem:duality-trunc}
Let $n, m \in \bN$ with $n \geq m$ and $\mathcal{P} \subseteq \TrBox{b,n}$ be a subgroup of cardinality $b^m$.
Then $\mathcal{P}$ is a $(t,m,d)$-net in $\Ucube_b$
if and only if
$\minNRT[n]{\mathcal{P}} \geq m-t+1$.
\end{lemma}
In order to extend the above for $\Pcal \subseteq \Ucube_b$, we need the following lemma.

\begin{lemma}\label{lem:NRT-equiv-trunc}
Let $\Pcal \subseteq \Ucube_b$ be a finite subgroup of cardinality $b^m$.
Then the following assertions are equivalent for all integers $n \geq m$:
\begin{enuroman}
\item \label{eq:trunc-i}
$\minNRT{\Pcal} \geq m-t+1$.
\item \label{eq:trunc-ii}
$\minNRT[n]{\trunc_{n}(\Pcal)} \geq m-t+1$.
\end{enuroman}
\end{lemma}

\begin{proof}
Lemma~\ref{lem:dnet} implies that $\Dnet[n]{\trunc_n(\Pcal)} \subset \Dnet{\Pcal}$. From this we immediately obtain the implication \eqref{eq:trunc-i}$\Rightarrow$\eqref{eq:trunc-ii}.

Now we assume \eqref{eq:trunc-ii} and prove \eqref{eq:trunc-i}.
We have $$\Dnet{\Pcal} = (\Dnet{\Pcal} \setminus \mathbb{M}_{b,n}^d) \cup (\Dnet{\Pcal} \cap  \mathbb{M}_{b,n}^d)=  (\Dnet{\Pcal} \setminus \mathbb{M}_{b,n}^d) \cup \Dnet[n]{\trunc_n(\Pcal)},$$ where the second equality follows from  Lemma~\ref{lem:dnet}. Therefore we obtain
\begin{align*}
\minNRT{\Pcal} 
&= \min \left(\minNRT[n]{\trunc_{n}(\Pcal)}, \inf_{\bsg \in \Dnet{\Pcal} \setminus \mathbb{M}_{b,n}^d} \mu(\bsg) \right)\\
&\geq \min(m-t+1, m+1) = m-t+1,
\end{align*}
where we used the assumption  \eqref{eq:trunc-ii} and the fact that $n \ge m$. 
\end{proof}

For a subgroup $\Pcal \subseteq \Ucube_b$ of cardinality $b^m$ let $$n := \max_{\bsg, \bsh \in \Pcal \atop \bsg \neq \bsh}\max_{1 \le j \le d}(-\deg(g_j-h_j)) \ge m,$$ where $\bsg=(g_1,\ldots,g_d)$ and $\bsh=(h_1,\ldots,h_d)$, so that $\trunc_n(\Pcal)$ has the cardinality $b^m$. Lemma~\ref{lem:duality-trunc} and Lemma~\ref{lem:NRT-equiv-trunc} for this $n$ imply the following duality theory for $\Pcal \subset \Ucube_b$.

\begin{theorem}\label{thm:duality-infty}
Let $\Pcal \subseteq \Ucube_b$ be a subgroup of cardinality $b^m$.
Then $\Pcal$ is a $(t,m,d)$-net in $\Ucube_b$
if and only if $\minNRT{\Pcal} \geq m-t+1$.
\end{theorem}

Now we finish the proof of Theorem~\ref{thm:t-value-formula}:

\begin{proof}[Proof of Theorem~\ref{thm:t-value-formula}]
Assume that $\bbX$ is admissible. From Lemma~\ref{lem:dnet} we obtain, that for any $\bsg = (g_1, \dots, g_d) \in \Dnet{\bbX \cap \Ucube_b} \setminus \{\bszero\}$, 
there exists $\bsg' = (g'_1, \dots, g'_d) \in \DinL{\bbX}$ with $\Ppart{\bsg'} = \bsg$.
If $g_j \neq 0$ then $$\mu(g_j) = 1 + \deg(g_j) = 1 + \deg(g'_j)$$
and if $g_j = 0$ then $$\mu(g_j) = 0 \geq 1 + \deg(g'_j).$$
Thus we have $$\mu(\bsg) \geq d+ \sum_{j=1}^d \deg(g'_j) \ \ \mbox{for all $\bsg \in \Dnet{\bbX \cap \Ucube_b}\setminus \{\bszero\}$.}$$
Hence
\begin{equation}\label{eq:minNRT-geq-adnum}
\minNRT{\bbX \cap \Ucube_b} 
\geq \adnum(\bbX) + d.
\end{equation}

For a shrunken lattice $\bsf^{-1}\bbX = \Diagm{\bsf}^{-1}T(\Pring^d)$ with 
$\bsf \in \Pring^d$ we note that
\[
((D_{\bsf}^{-1} T)^{\top})^{-1}= D_{\bsf} (T^{\top})^{-1}
\]
and hence every element from $\DinL{(\bsf^{-1}\bbX)}$ is of the form 
$(f_1 g_1,\ldots,f_d g_d)$ with $\bsg =(g_1,\ldots,g_d)\in\DinL{\bbX}$. Hence 
we have $$\adnum(\bsf^{-1}\bbX) = \adnum(\bbX) + \sum_{j=1}^{d} \deg (f_j)$$ and 
hence $\bsf^{-1}\bbX$ is also admissible.
Thus we can apply \eqref{eq:minNRT-geq-adnum} for $\bsf^{-1}\bbX$ and we have
\begin{align*}
\minNRT{\bsf^{-1}\bbX \cap \Ucube_b}
\geq \adnum(\bsf^{-1}\bbX) + d
= \sum_{j=1}^{d} \deg (f_j) + \adnum(\bbX) + d.
\end{align*}
Combining this result with Theorem~\ref{thm:number-of-points}, which says  that
\[
b^m=|\Ptset| = b^{- \deg(\det T) + \sum_{j=1}^{d} \deg (f_j)}
\]
for sufficiently large $\bsf$ satisfying \eqref{eq:condition-shrinking-factor}, we have
\begin{align*}
\minNRT{\bsf^{-1}\bbX \cap \Ucube_b}
\geq m - (- \deg(\det T) - \adnum(\bbX) - d + 1) + 1.
\end{align*}
Now, the result follows in conjunction with Theorem~\ref{thm:duality-infty}.
\end{proof}

\section{Construction}\label{sec:construction}
We now explicitly construct generating matrices $T$ when $d=b^n$ with $n \in \bN$.
Our construction is inspired by that for Frolov's original cubature formula, see~\cite{Fr76}.

Let $n \in \bN$, $d = b^n$ and 
\begin{equation}\label{eq:pd}
p_d := F_n + x^{-1} \in \Lfield[z]
\end{equation}
with
\begin{equation}\label{eq:Fn}
F_n(z) := \prod_{a_0,\ldots,a_{n-1}=0}^{b-1} (z-(a_0 + a_1 x + \dots + a_{n-1} x^{n-1})).
\end{equation}
We will show in Lemma~\ref{lem:roots-exist} that $p_d$ has $d$ different roots in $\Lfield$ which we denote by $\xi_1, \dots, \xi_{d} \in \Lfield$. Based on these roots we define the $d \times d$ Vandermonde matrix $B$ by
\begin{equation}\label{eq:B}
B=(B_{i,j})_{i,j=1}^d =(\xi_i^{j-1})_{i,j=1}^d. 
\end{equation}

\begin{theorem}\label{thm:construction}
Let $d=b^n$ for some $n\in\bN$ and $T = (B^{-1})^\top$, where $B$ is as in \eqref{eq:B}, 
be the generating matrix of the lattice $\bbX$. Then the lattice $\bbX$ is admissible.

Furthermore, if the shrinking factor $\bsf = (f_1, \dots, f_d) \in \Pring^d$ satisfies \eqref{eq:condition-shrinking-factor},
then $\Ptset$ is a $(t,m,d)$-net in $\Ucube_d$ with
\[
t \le  \frac{d}{2} \left(d \log_b d - (d-1) \frac{b}{b-1}\right) 
\ \ \mbox{ and }\ \
m =  \frac{d}{2} \left(d \log_b d - (d-1) \frac{b}{b-1}\right) \,+\,	\sum_{j=1}^{d} \deg f_j.
\]
\end{theorem}

\begin{remark}\rm
The condition $d=b^n$ can be removed. For general $d$ and given $b$, choose $n$ such that $d \le b^n=:d'$ and follow the above construction with dimension $d'$ instead of $d$ to obtain a $(t,m,d')$-net in $\mathbb{U}_b^{d'}$ with the parameters $m$ and $t$ as given in Theorem~\ref{thm:construction}. Finally, according to \cite[Lemma~4.16]{Dick2010dna}, the projection of this net to any of its $d$ components gives a $(t,m,d)$-net in $\Ucube_b$ with the same $t$- and $m$-parameter.      
\end{remark}

We give some further remarks:

\begin{remark}\rm
\begin{enumerate}
\item If $d\le b$, then we have $t=0$ and $m=\sum_{j=1}^{d} \deg f_j$. The quality parameter $t=0$ is best possible. Note that a $(0,m,d)$-net in base $b$ can only exist as long as $d \le b+1$, see \cite[Corollary~4.21]{Niederreiter1992rng} or \cite[Corollary~4.19]{Dick2010dna}.   
\item The $t$-values of the nets from our construction is of order $O(d^2 \log d)$.
However, there are known many $(t,m,d)$-nets whose $t$-values are less dependent on $d$.
For instance, the $t$-values of the Niederreiter sequence and the Niederreiter-Xing sequence
are of order $O(d \log d)$ and $O(d)$, respectively, see \cite[Theorem~4.54]{Niederreiter1992rng} and \cite[Theorem~8.4.4]{Niederreiter2001rpo}, respectively. 
From this point of view, our construction is asymptotically not optimal.
\end{enumerate}
\end{remark}

In what follows we provide some preparation for the proof of Theorem~\ref{thm:construction}.


\subsection{Auxiliary results}
In this subsection we indicate auxiliary results for polynomials in $\Lfield[z]$.
We use the following elementary computation without mentioning:
\begin{itemize}
\item $a^b = a \quad \text{for all $a \in \Fb$}$,
\item $(f + g)^b = f^b + g^b \quad \text{for all $f,g \in \Lfield[z]$}$,
\item $\prod_{a=0}^{b-1} (f+ag) = f^b - fg^{b-1} \quad \text{for all $f,g \in \Lfield[z]$}$.
\end{itemize}

\noindent
Some properties of the polynomial $F_n(z)$ are given by the following lemma.
\begin{lemma}\label{lem:Fn-recursive}
Let $F_n\in \Lfield[z]$ be as in \eqref{eq:Fn}. Then we have:  
\begin{enuroman}
\item $F_n(z) = F_{n-1}(z)^b - F_{n-1}(x^{n-1})^{b-1} F_{n-1}(z)$ for all $n \geq 2$, and \label{eq:Fn-rec1}
\item $F_n(z + af) = F_n(z) + aF_n(f)$ for all $n \in \bN$, $a \in \Fb$ and $f \in \Lfield$. \label{eq:Fn-rec2}
\end{enuroman}
\end{lemma}

\begin{proof}
We prove part \eqref{eq:Fn-rec2} of this lemma by induction on $n$.
For $n = 1$, we have
\[
F_1(z) := \prod_{a_0=0}^{b-1} (z-a_0) = z^b - z.
\]
Thus \eqref{eq:Fn-rec2} for $n=1$ holds.

Let $n \ge 2$. We now assume that \eqref{eq:Fn-rec2} holds for $n-1$ and prove the lemma for $n$.
We have
\begin{align}\label{pr:parti}
F_n(z)
&= \prod_{a_{n-1}=0}^{b-1} F_{n-1}(z -a_{n-1} x^{n-1})\nonumber\\
&= \prod_{a_{n-1}=0}^{b-1} \left(F_{n-1}(z) - a_{n-1} F_{n-1}(x^{n-1})\right)\nonumber\\
&= F_{n-1}(z)^b - F_{n-1}(x^{n-1})^{b-1} F_{n-1}(z),
\end{align}
where the induction hypothesis is used in the second equality.
Using this, we have
\begin{align*}
F_n(z + af)
&= F_{n-1}(z+af)^b - F_{n-1}(x^{n-1})^{b-1} F_{n-1}(z+af)\\
&= F_{n-1}(z)^b + a F_{n-1}(f)^b - F_{n-1}(x^{n-1})^{b-1} (F_{n-1}(z) + a F_{n-1}(f))\\
&= F_n(z) + a F_n(f),
\end{align*}
where we used the induction hypothesis a second time in the second equality.
This proves \eqref{eq:Fn-rec2} for $n$.

Finally, \eqref{eq:Fn-rec1} follows from \eqref{pr:parti}.
\end{proof}

In the following lemma we provide the roots of some polynomial.

\begin{lemma}\label{lem:roots-rep}
Let $f , g \in \Lfield$ with $f \neq 0$ and $\deg(f^{-b} g) < 0$.
Then the polynomial
\[
F(z) := z^b - z f^{b-1} + g
\]
has $b$ different roots in $\Lfield$ given by $\xi_a := (a + \sum_{i=0}^\infty (f^{-b} g)^{b^{i}}) f$ for $a \in \Fb$.
\end{lemma}
\begin{proof}
The condition $\deg(f^{-b} g) < 0$ ensures
that the infinite sum $\sum_{i=0}^\infty (f^{-b} g)^{b^{i}}$ converges in $\Lfield$ and thus $\xi_a \in \Lfield$. We have
\begin{align*}
\xi_a^b & =   \left(a + \sum_{i=0}^\infty (f^{-b} g)^{b^{i}}\right)^b f^b = \left(a+\sum_{i=1}^\infty (f^{-b} g)^{b^{i}}\right) f^b\\
&= \left(\left(a+\sum_{i=0}^\infty (f^{-b} g)^{b^{i}} - f^{-b} g\right)f^{-1} f  \right) f^b = \xi_a f^{b-1} -g
\end{align*}
and hence it follows that $F(\xi_a)=0$ for all $a \in \Fb$.
\end{proof}

The elementary symmetric polynomials in $n$ variables $x_1, \dots, x_n$,
written $e_k(x_1, \dots, x_n)$ for $0 \leq k \leq n$ are defined by
$e_0(x_1, \dots, x_n)=1$ and, for $1 \leq k \leq n$,
\[
e_k(x_1, \dots, x_n) := \sum_{1 \leq j_1 < j_2 < \ldots < j_k \leq n}x_{j_1} \cdots x_{j_k}.
\]

The fundamental theorem on symmetric polynomials is well known, see, e.g., \cite[Section~I.2]{Macdonald1995sfh}.

\begin{lemma}[Fundamental theorem on symmetric polynomials]\label{lem:sympoly}
Let $R$ be a commutative ring,
$R[x_1, \dots, x_n]$ the ring of polynomials,
and $R^{{\rm sym}}[x_1, \dots, x_n]$ the ring of symmetric polynomials in the variables $x_1, \dots, x_n$ with coefficients in $R$. Then, for every symmetric polynomial $P \in R^{{\rm sym}}[x_1, \dots, x_n]$, there exists a unique polynomial $Q \in R[y_1, \dots, y_n]$ such that
\[
P = Q(e_1(x_1, \dots, x_n), \dots, e_n(x_1, \dots, x_n)).
\]
\end{lemma}

Eisenstein's criterion is used to show the irreducibility of polynomials, see, e.g., \cite[Chapter~V, \S7]{Lang1965a}.

\begin{lemma}[Eisenstein's criterion]\label{lem:Eisenstein}
Let $R$ be a unique factorization domain and $k$ be its quotient field.
Let
\[
f(z) = z^n + a_{n-1}z^{n-1} + a_{n-2}z^{n-2} + \cdots + a_1z + a_0
\]
be a polynomial  of degree $n \geq 1$ in $R[z]$.
Let $p$ be a prime in $R$ such that $p$ divides $a_i$ for all $0 \leq i \leq n-1$
but $p^2$ does not divide $a_0$.
Then $f$ is irreducible in $k[z]$.
\end{lemma}

\subsection{Proof of Theorem \ref{thm:construction}}
Let $n \in \bN$, $p_d= F_n+ x^{-1}$ as in \eqref{eq:pd} where $F_n$ is defined as in \eqref{eq:Fn}. In a first step we determine the roots of $p_d$ in a more general setting.

\begin{lemma}\label{lem:roots-exist}
Let $n \in \bN$.
Let $f \in \Lfield$ with $\deg(f) < nb^n - (b^{n+1}-b)/(b-1)$.
Then the polynomial $F_n + f$
has $d=b^n$ different roots which are given by
$$a_{n-1} x^{n-1} + \cdots + a_1 x + a_0 + \xi$$
with $a_i \in \mathbb{F}_b$ for all $0 \leq i \leq n-1$ and some $\xi \in \Lfield$ with $\deg(\xi) = \deg(f) - nb^n + (b^{n+1}-b)/(b-1)$.
\end{lemma}

\begin{proof}
It suffices to prove
that there exists $\xi \in \Lfield$ with $\deg(\xi) = \deg(f) - nb^n + (b^{n+1}-b)/(b-1)$
such that $F_n(\xi) + f = 0$,
since for such $\xi$ we have
$F_n(a_{n-1} x^{n-1} + \cdots + a_1 x + a_0 + \xi) + f = 0$ for all $a_i \in \mathbb{F}_b$ for all  $0 \leq i \leq n-1$.

Now we prove the lemma by induction on $n$.
The case $n=1$ is a direct consequence of Lemma~\ref{lem:roots-rep}.
Let now $n \ge 2$ and assume that the assertion holds for $n-1$. We prove the lemma for $n$.
Lemmas~\ref{lem:Fn-recursive} and \ref{lem:roots-rep} imply
\[
F_n(z) + f
= F_{n-1}(z)^b - F_{n-1}(x^{n-1})^{b-1} F_{n-1}(z) + f
= \prod_{a=0}^{b-1} (F_{n-1}(z) - {f_a}),
\]
where
$f_a := F_{n-1}(x^{n-1}) (a+ \sum_{i=0}^\infty (F_{n-1}(x^{n-1})^{-b} f)^{b^{i}})$
and the assumption in Lemma~\ref{lem:roots-rep} is satisfied since
$$\deg(F_{n-1}(x^{n-1})^{-b} f) = \deg(f) - b^n(n-1) < n b^n - \frac{b^{n+1}-b}{b-1}-n b^n+b^n=\frac{b-b^n}{b-1}\le 0.$$
We consider the polynomial $F_{n-1}(z) - {f_0}$.
The degree of ${f_0}$ is given by
\begin{align}
\deg({f_0})
&= \deg(F_{n-1}(x^{n-1})^{-b+1}) + \deg(f) \label{eq:deg-f0}\\
&< -(n-1)(b-1)b^{n-1} + nb^n - \frac{b^{n+1}-b}{b-1} \notag \\
&= (n-1)b^{n-1} - \frac{b^n-b}{b-1}. \notag
\end{align}
Hence we can apply the induction assumption for $n-1$ to the polynomial $F_{n-1}(z) - {f_0}$,
which implies that it has a root $\xi$ with
\[
\deg(\xi) = \deg({f_0}) - (n-1)b^{n-1} + \frac{b^{n}-b}{b-1} = \deg(f) - nb^n + \frac{b^{n+1}-b}{b-1},
\]
where the last equality follows from \eqref{eq:deg-f0}.
This is the $\xi$ whose existence we intended to show.
\end{proof}

The following lemma shows the irreducibility of $p_d$ in $\Fb(x)[z]$, where $\Fb(x)$ is the field of rational 
functions of the form $f/g$ with $f,g \in \Fb[x]$, $g \not=0$. 
This property implies that $p_d$ is the minimal polynomial of $\xi_j$ over $\Fb[x]$ 
for all $1 \le j \le d$. 

\begin{lemma}\label{le:irred}
For all $n \in \bN$ the polynomial $p_d = F_n + x^{-1} \in \Fb(x)[z]$
is irreducible.
\end{lemma}

\begin{proof}
Put $c := b^n-1$. It suffices to show that
$x^{(c+1)/c}p_d(z)$ is irreducible in $\Fb(x^{1/c})[z]$.
We have
\begin{align}
x^{(c+1)/c}p_d(z)
&= \prod_{a_0,\ldots,a_{n-1}=0}^{b-1} (x^{1/c}z-x^{1/c}(a_0 + a_1 x + \dots + a_{n-1} x^{n-1})) + x^{1/c} \notag\\
&= \prod_{a_0,\ldots,a_{n-1}=0}^{b-1} (z'-x^{1/c}(a_0 + a_1 x + \dots + a_{n-1} x^{n-1})) + x^{1/c}, \label{eq:Irrlem}
\end{align}
where we put $z' := x^{1/c}z$.
Thus it reduces to show the irreducibility of \eqref{eq:Irrlem} in $\Fb(x^{1/c})[z']$.
This follows by Eisenstein's criterion (Lemma~\ref{lem:Eisenstein}) with $p = x^{1/c}$.
\end{proof}

Now we determine the degree of the determinant of the matrix $B$ from \eqref{eq:B}.

\begin{lemma}\label{lem:detB}
Let $n \in \bN$ and let $B$ be the $b^n \times b^n$ matrix from \eqref{eq:B}.
Then 
\[
\deg(\det B) = \frac{b^n}{2} \left((n-1)b^n - \frac{b^n - b}{b-1}\right).
\]
\end{lemma}

\begin{proof}
Lemma~\ref{lem:roots-exist} implies that for fixed $i \in \{1, \ldots,b^n\}$ we have
\begin{align*}
\deg \left(\prod_{\substack{j=1 \\ j \neq i}}^{b^n}(\xi_j -\xi_i)\right)
&= \sum_{\substack{a_0,\ldots,a_{n-1}=0\\ (a_0,\ldots,a_{n-1})\not=(0,\ldots,0)}}^{b-1} \deg(a_{n-1} x^{n-1} + \cdots + a_1 x + a_0)\\
&= (n-1)b^n - \frac{b^n - b}{b-1}.
\end{align*}		
Thus
\begin{align*}
\deg(\det B)
&= \frac{1}{2} \deg\left(\prod_{\substack{i,j =1\\ j\not= i}}^{b^n} (\xi_j -\xi_i)\right)= \frac{b^n}{2}  \left((n-1)b^n - \frac{b^n - b}{b-1}\right).
\end{align*}
where the first equality follows form the well known result for the determinant of a Vandermonde matrix. This proves the result.
\end{proof}

The following lemma is needed to show the admissibility of our lattice.

\begin{lemma}\label{lem:admissible-bn}
Let $d=b^n$ and $\bszeta = (\zeta_1, \dots, \zeta_d) \in B(\Pring^d)\setminus\{\bszero\}$ with $B$ from \eqref{eq:B}.
Then $\deg(\prod_{j=1}^d \zeta_j) \geq 1-d$ holds.
In particular, the lattice $\bbX$ constructed as in Theorem~\ref{thm:construction} is admissible with $\adnum(\bbX) \geq 1-d$.
\end{lemma}

\begin{proof}
Fix $\bsf = (f_1, \dots, f_d) \in \Pring\setminus\{\bszero\}$ such that $B\bsf = \bszeta$.
Then $$\zeta_j = f_1 + f_2 \xi_j + \cdots + f_d \xi_j^{d-1}.$$

First we show that $\zeta_j \neq 0$ for all $1 \le j \le d$.
According to Lemma~\ref{le:irred} the polynomial $p_d$ is the minimal polynomial of $\xi_j$ for all $1 \le j \le d$. 
Hence there is no polynomial $r \in \Fb(x^{-1})[z]$ with $\deg(r) < d$ and $r(\xi_j) = 0$.
This implies that $\zeta_j \neq 0$.

We now put $p := \prod_{j=1}^d \zeta_j$
and show that $\deg(p) \geq 1-d$.
Note that $\zeta_j \neq 0$ for all $j$ implies $p \neq 0$.
Since $p$ is a symmetric polynomial in $\xi_1, \dots, \xi_d$ with coefficients in $\Pring$,
there exists a polynomial $q \in \Pring[y_1, \dots, y_d]$
such that $p = q(e_1, \dots, e_d)$ where $e_i := e_i(\xi_1, \dots, \xi_d)$, see Lemma~\ref{lem:sympoly}.
Comparing the degree of $\xi_j$'s on both sides,
we have that the multiplicity of $e_d$ in $q(e_1, \dots, e_d)$ is at most $d-1$.
Further, Vieta's formulas imply that
\begin{equation} \label{eq:coeff-vieta}
e_i = e_i(\xi_1, \dots, \xi_d) = 
\begin{cases}
(-1)^{d} x^{-1} & \text{if $i=d$},\\
\text{some element in $\Pring$} & \text{otherwise}.
\end{cases}
\end{equation}
Hence each term in $q(e_1, \dots, e_d)$, and thus $p=q(e_1, \dots, e_d)$ itself,
can by written as $g x^{1-d}$ with some $g \in \Pring$. 
This and $p \neq 0$ prove $\deg(p) \geq 1-d$.
\end{proof}

\begin{proof}[Proof of Theorem~\ref{thm:construction}]
The result is a direct consequence of Theorem~\ref{thm:t-value-formula} combined with Lemmas~\ref{lem:detB} and \ref{lem:admissible-bn}. 
\end{proof}



\goodbreak


\noindent\textbf{Authors' addresses:}\\

\noindent {\sc Josef Dick}

\noindent School of Mathematics and Statistics, The University of New South Wales, Sydney NSW 2052, Australia, 
email: josef.dick(AT)unsw.edu.au\\

\noindent {\sc Friedrich Pillichshammer}

\noindent Institut f\"ur Finanzmathematik und Angewandte Zahlentheorie, Johannes Kepler Universit\"at Linz, Altenbergerstra{\ss}e 69, 4040 Linz, Austria,
email: friedrich.pillichshammer(AT)jku.at\\

\noindent {\sc Kosuke Suzuki}

\noindent 
Graduate School of Science, Hiroshima University.
1-3-1 Kagamiyama, Higashi-Hiroshima, 739-8526, Japan,
email: kosuke-suzuki(AT)hiroshima-u.ac.jp\\

\noindent {\sc Mario Ullrich}

\noindent Institut f\"ur Analysis, Johannes Kepler Universit\"at Linz, Altenbergerstra{\ss}e 69, 4040 Linz, Austria, 
email: mario.ullrich(AT)jku.at\\

\noindent{\sc Takehito Yoshiki} 

\noindent{Department of Applied Mathematics and Physics, Graduate School of Informatics, Kyoto University, Kyoto 606-8561, Japan, email: yoshiki.takehito.47x(AT)st.kyoto-u.ac.jp}
\end{document}